\documentclass[reqno,a4paper,final]{amsart}

\usepackage[utf8]{inputenc}
\usepackage[T1]{fontenc}
\usepackage[english]{babel}
\usepackage{ifthen}
\usepackage{cite}
\usepackage{enumitem}
\usepackage{amsmath}
\usepackage{amsthm}
\usepackage{amsfonts}
\usepackage{amssymb}
\usepackage{dsfont}
\usepackage{mathtools}
\usepackage[unicode,bookmarksopen]{hyperref}
\numberwithin{equation}{section} 

\newtheorem{lemma}{Lemma}[section]
\newtheorem{corollary}[lemma]{Corollary}
\newtheorem{proposition}[lemma]{Proposition}
\newtheorem{theorem}[lemma]{Theorem}

\theoremstyle{definition}
\newtheorem{definition}[lemma]{Definition}

\newtheorem{remark}[lemma]{Remark}

\newlist{thm_enum}{enumerate}{1}
\setlist[thm_enum]{label=\normalfont(\alph*)}
\newlist{def_enum}{enumerate}{1}
\setlist[def_enum]{label=\normalfont(\roman*)}
\newlist{equiv_enum}{enumerate}{1}
\setlist[equiv_enum]{label=\normalfont(\roman*)}

\newcommand{\IQ}{\mathbb{Q}}
\newcommand{\IN}{\mathbb{N}}
\newcommand{\IR}{\mathbb{R}}
\newcommand{\IC}{\mathbb{C}}

\renewcommand{\epsilon}{\varepsilon}
\renewcommand{\phi}{\varphi}
\newcommand{\abs}[1]{\left\lvert#1\right\rvert}

\newcommand{\norm}[1]{\left\lVert#1\right\rVert}
\newcommand{\biggnorm}[1]{\biggl\lVert#1\biggr\rVert}

\newcommand{\R}[2][\empty]{
	\ifthenelse{\equal{#1}{\empty}}
		{\mathcal{R}\left\{#2\right\}}
		{\mathcal{R}_{#1}\left\{#2\right\}}
}

\renewcommand{\Im}{\operatorname{Im}}

\DeclareMathOperator{\linspan}{span}

\DeclareMathOperator{\Id}{Id}

\DeclareMathOperator{\diag}{diag}

\allowdisplaybreaks[4]

\begin{document}

\title{On the structure of semigroups on $L_p$ with a bounded $H^{\infty}$-calculus}

\begin{abstract}
	We show that a bounded analytic semigroup on a reflexive $L_p$-space has a bounded $H^{\infty}(\Sigma_{\theta})$-calculus for some $\theta < \frac{\pi}{2}$ if and only if the semigroup can be obtained, after restricting to invariant subspaces, factorizing through invariant subspaces and similarity transforms, from a bounded analytic semigroup on some bigger $L_p$-space which is positive and contractive on the real line.
\end{abstract}

\author{Stephan Fackler}
\address{Institute of Applied Analysis, University of Ulm, Helmholtzstr. 18, 89069 Ulm}
\email{stephan.fackler@uni-ulm.de}
\thanks{The author was supported by a scholarship of the ``Landesgraduiertenförderung Baden-Württemberg''. Moreover, he wants to thank the anonymous referee for his valuable suggestions.}

\maketitle

\section{Introduction}\label{sec:introduction}
	It is a central task in operator theory to understand the structure of various functional calculi associated to certain classes of operators. In this work we study the structure of the holomorphic functional calculus for pseudo-sectorial operators on reflexive $L_p$-spaces. Here a densely defined linear operator $A\colon X \supset D(A) \to X$ with domain $D(A)$ on a Banach space $X$ is called \emph{pseudo-sectorial} if there exists an $\omega \in (0, \pi)$ such that
	\begin{equation}
		\label{sectorial}
		\tag{$S_{\omega}$}
		\sigma(A) \subset \overline{\Sigma_{\omega}} \qquad \text{and} \qquad \sup_{\lambda \not\in \overline{\Sigma_{\omega + \epsilon}}} \norm{\lambda (\lambda - A)^{-1}} < \infty \quad \text{for all } \epsilon > 0,
	\end{equation}
	where $\Sigma_{\omega} \coloneqq \{ z \in \IC \setminus \{0\}: \abs{\arg z} < \omega \}$ is the open sector in the complex plane of opening angle $\omega$ centered on the real line and $\sigma(A)$ denotes the spectrum of $A$. One defines the \emph{sectorial angle of $A$} as $\omega(A) \coloneqq \inf \{ \omega: \text{\eqref{sectorial} holds} \}$.
	
	Pseudo-sectorial operators have a natural holomorphic functional calculus. For $\theta \in (0, \pi)$ let
		\begin{align*} 
			& H_0^{\infty}(\Sigma_{\theta}) \coloneqq \left\{ f\colon \Sigma_{\theta} \to \IC \text{ analytic}\colon \abs{f(\lambda)} \le C \frac{\abs{\lambda}^{\epsilon}}{(1 + \abs{\lambda})^{2\epsilon}} \text{ on } \Sigma_{\theta} \text{ for } C, \epsilon > 0 \right \}, \\
			& H^{\infty}(\Sigma_{\theta}) \coloneqq \{ f\colon \Sigma_{\theta} \to \IC \text{ analytic and bounded} \}
		\end{align*}
	endowed with the norm of uniform convergence. Then for a pseudo-sectorial operator $A$ on some Banach space $X$ and $f \in H_0^{\infty}(\Sigma_{\theta})$ we define
		\[ f(A) = \frac{1}{2\pi i} \int_{\partial \Sigma_{\theta'}} f(\lambda) (\lambda - A)^{-1} \, d\lambda \qquad (\omega(A) < \theta' < \theta).  \]
	This is well-defined by the growth estimate on $f$ and by the invariance of the contour integral and induces an algebra homomorphism from $H^{\infty}_0(\Sigma_{\theta})$ into the Banach algebra $\mathcal{B}(X)$ of all bounded operators on $X$. We say that $A$ has a \emph{bounded $H^{\infty}(\Sigma_{\theta})$-calculus} for some $\theta > \omega(A)$ if there is a constant $C_{\theta} \ge 0$ such that
		\[ \norm{f(A)}_{\mathcal{B}(X)} \le C_{\theta} \sup_{z \in \Sigma_{\theta}} \abs{f(z)} \qquad \text{for all } f \in H^{\infty}_0(\Sigma_{\theta}). \] 
	The infimum of those $\theta$ for which the above estimate holds for some $C_{\theta} \ge 0$ is denoted by $\omega_{H^{\infty}}(A)$. The concept of a bounded $H^{\infty}$-calculus goes back to~\cite{McI86} for Hilbert spaces and its study was extended to Banach spaces in~\cite{CDMY96}. Moreover, the terminology is justified by the fact that if $A$ additionally has dense range -- in this case we call $A$ \emph{sectorial} -- the functional calculus can be extended to $H^{\infty}(\Sigma_{\theta})$ (see Section~\ref{sec:functional_calculus}). 
	
	Both for theory and applications it is important to determine general classes of pseudo-sectorial operators which have a bounded $H^{\infty}$-calculus. Such classes are given by certain analytic semigroups, i.e. strongly continuous mappings $z \mapsto T(z)$ with $T(0) = \Id$ from a sector $\Sigma \cup \{0\}$ into $\mathcal{B}(X)$ for some Banach space $X$ satisfying the semigroup identity $T(z_1 + z_2) = T(z_1)T(z_2)$ for all $z_1, z_2 \in \Sigma$. Note that analytic semigroups on a Banach space $X$ are in one-to-one correspondence with pseudo-sectorial operators on $X$ of angle smaller than $\frac{\pi}{2}$, which is formally given by $A \mapsto (e^{-zA})$. In this case the operator $-A$ corresponding to $(T(z))$ is called the generator of the semigroup $(T(z))$. 
	
	A famous criterion states that the negative generator of a bounded analytic semigroup $(T(z))$ on a Hilbert space $H$ with $\norm{T(t)} \le 1$ for all $t \ge 0$ has a bounded $H^{\infty}(\Sigma_{\theta})$-calculus for some $\theta < \frac{\pi}{2}$. Conversely, every semigroup $(T(t))_{t \ge 0}$ on $H$ whose negative generator has a bounded $H^{\infty}(\Sigma_{\theta})$-calculus for some $\theta < \frac{\pi}{2}$ is similar to a contractive semigroup, i.e. there exists an invertible operator $S \in \mathcal{B}(H)$ such that $\norm{S^{-1} T(t) S} \le 1$ for all $t \ge 0$. In other words, there exists an isomorphic Hilbert space such that the semigroup is contractive on the real line in this Hilbert space. This semigroup variant of the famous Halmos similarity problem was proved by C.~Le Merdy~\cite{Mer98}.
	
	Leaving the Hilbert space setting, on $L_p$ for $p \in (1, \infty)$ we still have an important class of analytic semigroups whose negative generators have a bounded $H^{\infty}(\Sigma_{\theta})$-calculus for some $\theta < \frac{\pi}{2}$: by~\cite[Remark~4.9c)]{Wei01} the negative generator of a bounded analytic semigroup on $L_p$ that is contractive and positive in the natural order of $L_p$ on the positive real line has a bounded $H^{\infty}(\Sigma_{\theta})$-calculus for some $\theta < \frac{\pi}{2}$. 
	
	In this paper we show the following converse of Weis' result, generalizing Le Merdy's theorem from the Hilbert space case to an $L_p$-setting: for a semigroup $(T(t))_{t \ge 0}$ on $L_p$ ($p \in (1, \infty)$) whose negative generator has a bounded $H^{\infty}(\Sigma_{\theta})$-calculus for some $\theta < \frac{\pi}{2}$ there exist a bounded analytic semigroup $(R(z))$ on some bigger $L_p$-space, which is positive and contractive on the real line, $(R(z))$-invariant closed subspaces $N \subset M$ and an isomorphism $S$ from $L_p$ onto $M/N$ such that $T(t) = S^{-1} \hat{R}(t) S$ for all $t \ge 0$, where $(\hat{R}(t))_{t \ge 0}$ is the induced semigroup on $M/N$. An analogous result is also obtained in Theorem~\ref{thm:generic} for semigroups on subspace-quotients of general UMD-Banach lattices (UMD stands for unconditional martingale difference property).
	
	In other words, the class of bounded analytic semigroups on $L_p$ which are positive and contractive on the real line is generic for all semigroups on $L_p$ whose negative generators have a bounded $H^{\infty}(\Sigma_{\theta})$-calculus for some $\theta < \frac{\pi}{2}$. That is, modulo the operations (1) similarity transforms, (2) passing to an invariant subspace and (3) factorizing through an invariant subspace,
	which all preserve the boundedness of the functional calculus, all such semigroups are bounded analytic semigroups being positive and contractive on the real line.
	
\section{\texorpdfstring{$H^{\infty}$-}{Holomorphic functional }calculus}\label{sec:functional_calculus}

We give further comments on sectorial operators and their associated holomorphic functional calculus. For more details we refer to the presentations in~\cite{KunWei04} and \cite{DHP03}. 			
	
	Let $A$ be a pseudo-sectorial operator with $\omega(A) < \frac{\pi}{2}$ and $(T(z))$ be the analytic semigroup generated by $-A$. Then one can plug $f_z\colon \lambda \mapsto e^{-\lambda z} - 1/(1+\lambda)$ into the functional calculus for $z \in \Sigma_{\pi/2 - \omega(A)}$. For such $z$ one can show that $f_z(A) = T(z) - (1 + A)^{-1}$ (this follows, for example, from~\cite[Theorem~9.6]{KunWei04}), that is the functional calculus reproduces the bounded analytic $C_0$-semigroup generated by $-A$. 
	
	Now let $A$ be a sectorial operator with a bounded $H^{\infty}(\Sigma_{\theta})$-calculus, which means that we additionally assume that $A$ has dense range. Then the holomorphic functional calculus defined above can be continued to a Banach algebra homomorphism $H^{\infty}(\Sigma_{\theta}) \to \mathcal{B}(X)$. Moreover, this extension is unique under some mild regularity assumptions~\cite[Remark~9.7]{KunWei04}. The extension can be constructed as follows~\cite[§9]{KunWei04}:

For $n \ge 2$ let $\rho_n(\lambda) \coloneqq \frac{n}{n+\lambda} - \frac{1}{1 + \lambda n} \in H_0^{\infty}(\Sigma_{\theta})$. Then one has $\rho_n(A) = n(n+A)^{-1} - \frac{1}{n}(\frac{1}{n} + A)^{-1}$, which are uniformly bounded in $n \in \IN$ by the pseudo-sectoriality of $A$. Moreover, for all $f \in H^{\infty}(\Sigma_{\theta})$ one has $(f\rho_n)(\lambda) \to f(\lambda)$ for all $\lambda \in \Sigma_{\theta}$ and $\rho_n(A)x \to x$ for all $x \in X$. The extension can then be defined as
	\[ f(A)x \coloneqq \lim_{n \to \infty} (f \rho_n)(A)x \quad \text{for all } x \in X. \]

If $X$ is reflexive, then one can always decompose a pseudo-sectorial operator $A$ as $A = \begin{pmatrix} A_{00} & 0 \\ 0 & 0 \end{pmatrix}$ with respect to the decomposition $X = \overline{R(A)} \oplus N(A)$ such that $A_{00}$ is a sectorial operator on $\overline{R(A)}$. In this case $Px = \lim_{n \to \infty} -\frac{1}{n} R(-\frac{1}{n},A)x$ is the projection onto the null space $N(A)$.

If $-A$ is the generator of a bounded $C_0$-semigroup $(T(t))_{t \ge 0}$ and one has $\theta > \frac{\pi}{2}$, the functional calculus can be written in terms of the semigroup. Indeed, if $f \in H_0^{\infty}(\Sigma_{\theta})$ there exists a unique $b \in L^1(\IR_+)$ such that
	\[ f = \hat{b} \qquad \text{and} \qquad f(A) = \int_0^{\infty} b(t) T(t) \, dt \]
in the strong operator topology, where $\hat{b}$ is the Laplace transform of $b$ defined on $\overline{\Sigma_{\pi/2}}$ by
	\[ \hat{b}(z) = \int_0^{\infty} b(t) e^{-zt} \, dt. \]
Then~\cite[Proposition~2.13]{Mer99} the pseudo-sectorial operator $A$ admits a bounded $H^{\infty}(\Sigma_{\theta})$-calculus for any $\theta > \frac{\pi}{2}$ if and only if there exists a constant $C \ge 0$ such that for every $b \in L^1(\IR_+)$ whose Laplace transform $\hat{b}$ belongs to $H_0^{\infty}(\Sigma_{\theta})$ one has
	\[ \biggnorm{\int_0^{\infty} b(t) T(t) \, dt} \le C \lVert \hat{b} \rVert_{H^{\infty}(\Sigma_{\theta})}. \]

\section{\texorpdfstring{$p$}{p}-Completely Bounded Maps and \texorpdfstring{$p$}{p}-Matrix Normed Spaces}

In this section, we present the needed operator space theoretic background on completely bounded maps and matrix normed spaces. As we do not want to delve deeply in this branch of functional analysis, we give an ad-hoc definition only introducing the terminology necessary for this article following~\cite{LeM96}. Let $E, F$ be two Banach spaces which are embedded into the algebras of bounded operators $\mathcal{B}(X)$ and $\mathcal{B}(Y)$ of two Banach spaces $X$ and $Y$. A linear map $u\colon E \to F$ induces (for $n \in \IN$) a linear map
	\begin{align*} 
		u_n\colon M_n(\mathcal{B}(X)) \supset M_n(E) & \to M_n(F) \subset M_n(\mathcal{B}(Y)) \\
		[ a_{ij} ]  & \mapsto [u(a_{ij})]	
	\end{align*}
between the matrix algebras. For a fixed $p \in [1, \infty]$ the algebra $M_n(\mathcal{B}(X))$ can be identified with $\mathcal{B}(\ell_p^n(X))$. For $p < \infty$ the norm of a matrix element $[a_{ij}] \in M_n(E)$ is then given as
	\[ \norm{[a_{ij}]}_{M_n(E)}^p = \sup \left\{ \sum_{i=1}^n \biggnorm{\sum_{j=1}^n a_{ij}(x_j)}^p : \sum_{j=1}^n \norm{x_j}_{X}^p \le 1 \right \}. \]
	
\begin{definition}[Complete Boundedness] A map $u\colon E \to F$ as above is \emph{$p$-completely bounded} ($p \in [1,\infty]$) if the induced maps $u_n\colon M_n(E) \to M_n(F)$ seen as linear maps between subspaces of $\mathcal{B}(\ell_p^n(X))$ and $\mathcal{B}(\ell_p^n(Y))$ are uniformly bounded in $n \in \IN$.
\end{definition}

Note that the $p$-complete boundedness depends on the choice of the embeddings $E \hookrightarrow \mathcal{B}(X)$ and $F \hookrightarrow \mathcal{B}(Y)$. We will call the datum of a Banach space $E$ with an embedding into $\mathcal{B}(X)$ a \emph{$p$-matrix normed space structure} for $E$ and we will always consider $\mathcal{B}(Z)$ for a Banach space $Z$ with its natural $p$-matrix normed space structure.

\section{The \texorpdfstring{$p$}{p}-Matrix Normed Space Structure for \texorpdfstring{$H^{\infty}(\Sigma_{\theta})$}{ The Algebra of Bounded Holomorphic Functions}}

Let $Y$ be a Banach space and let $V(t)g(s) = g(s - t)$ be the shift group on $L_p(\IR; Y)$ for $p \in (1, \infty)$ and $B$ be its negative infinitesimal generator. Then for $\theta > \frac{\pi}{2}$ one has $f(B)g = b * g$ for $f \in H_0^{\infty}(\Sigma_{\theta})$, where $b$ is the unique element in $L^1(\IR_+)$ such that $f = \hat{b}$, the Laplace transform of $b$. The Banach space valued variant of the Mikhlin multiplier theorem~\cite[Proposition~3]{Zim89} then shows that $B$ has a bounded $H^{\infty}$-calculus with $\omega_{H^{\infty}}(B) = \frac{\pi}{2}$ if $Y$ is a UMD-space. 

UMD-spaces provide a natural setting for vector-valued harmonic analysis. We refer the reader to~\cite{Fra86} and~\cite{Bur01} for details. We only note for further usage that for $p \in (1, \infty)$ the Bochner--Lebesgue spaces $L_p(X)$ are UMD-spaces if $X$ is a UMD-space and that the UMD-property is stable under subspaces and quotients. In particular, the $L_p$-spaces are UMD for $p \in (1, \infty)$. Furthermore, UMD-spaces are always reflexive.

Using the boundedness of the $H^{\infty}(\Sigma_{\theta})$-calculus of $B$ on $L_p(\IR;Y)$ for each $\theta > \frac{\pi}{2}$ one can define an embedding of $H^{\infty}(\Sigma_{\frac{\pi}{2}+}) \coloneqq \bigcup_{\theta > \frac{\pi}{2}} H^{\infty}(\Sigma_{\theta})$ (as vector spaces)
	\begin{equation*}
		\label{OS1}
		\tag{$p$-MNS1}
		\begin{split}
			H^{\infty}(\Sigma_{\frac{\pi}{2}+})  & \hookrightarrow \mathcal{B}(L_p(\IR;Y)) \\
			f & \mapsto f(B). 
		\end{split}
	\end{equation*}
The above map is indeed injective: First let $f_1(B) = f_2(B)$ with $f_i \in H_0^{\infty}(\Sigma_{\theta})$ for some $\theta > \frac{\pi}{2}$. This implies for the inverse Laplace transforms $b_i$ of $f_i$ that $b_1 * g = b_2 * g$ for all $g \in L_p(\IR; Y)$ and therefore $b_1 = b_2$, which yields $f_1 = f_2$. For the general case we use the fact that $H_0^{\infty}(\Sigma_{\frac{\pi}{2}+})$ is an ideal in $H^{\infty}(\Sigma_{\frac{\pi}{2}+})$: $f_1(B) = f_2(B)$ implies $(f_1 \rho)(B) = (f_2 \rho)(B)$, where $\rho(\lambda) = \frac{\lambda}{(1+ \lambda)^2}$, which in turn shows $f_1 \rho = f_2 \rho$ and therefore $f_1 = f_2$.

We now endow $H^{\infty}(\Sigma_{\frac{\pi}{2}+})$ with the norm induced as a subspace of $\mathcal{B}(L_p(\IR;Y))$. Note that this also gives $H^{\infty}(\Sigma_{\frac{\pi}{2}+})$ the structure of a $p$-matrix normed space. We will call this $p$-matrix normed space structure the \emph{$p$-matrix normed space structure~\eqref{OS1} with respect to $Y$}.

The above choice of a matrix normed space structure is natural in view of transference techniques, but has the disadvantage that it does not make use of the angle $\omega_{H^{\infty}}$ and therefore loses information on the strength of the functional calculus. We will now solve this issue by using the fractional powers of the generator $B$ to define refined versions of the above embedding. For this we use the following proposition.

\begin{proposition}\label{prop:fractional_powers}
	Let $A$ be a sectorial operator on a Banach space. Then for $\alpha \in (0, 2\pi / \omega(A))$ the fractional powers $A^{\alpha}$ are sectorial with $\omega(A^{\alpha}) = \alpha \omega(A)$. Moreover, if $A$ has a bounded $H^{\infty}(\Sigma_{\theta})$-calculus and $\theta \in (0, 2\pi / \omega_{H^{\infty}}(A))$, then $A^{\alpha}$ has a bounded $H^{\infty}$-calculus with $w_{H^{\infty}}(A^{\alpha}) = \alpha \omega_{H^{\infty}}(A)$.
\end{proposition}
\begin{proof}
	A proof of the first part can be found in~\cite[Theorem~15.16]{KunWei04}, whereas the second part follows from~\cite[Proposition~15.11]{KunWei04}.
\end{proof}

In the setting as above choose $\alpha$ with $\alpha \in [1, \infty)$. Then $B^{1/\alpha}$ has a bounded $H^{\infty}(\Sigma_{\frac{\pi}{2\alpha}+})$-calculus. Now consider the embedding
	\begin{equation*}
		\label{OS2}
		\tag{$p$-MNS2}
		\begin{split}
			H^{\infty}(\Sigma_{\frac{\pi}{2\alpha}+}) & \hookrightarrow \mathcal{B}(L_p(\IR;Y)) \\
			f & \mapsto f(B^{\frac{1}{\alpha}}).
		\end{split}
	\end{equation*}
Notice that one has $f(B^{\frac{1}{\alpha}}) = (f \circ \cdot^{1/\alpha})(B)$ and that for $\theta \in (0, \pi)$ one has an isomorphism
	\begin{align*}
		H^{\infty}(\Sigma_{\theta}) & \to H^{\infty}(\Sigma_{\theta/\alpha}) \\
		f & \mapsto [ \lambda \mapsto f(\lambda^{\alpha}) ].
	\end{align*}
Therefore $f \circ \cdot^{1/\alpha} \in H^{\infty}(\Sigma_{\frac{\pi}{2}+})$ and injectivity follows from the case $\alpha = 1$ considered above. This again endows $H^{\infty}(\Sigma_{\frac{\pi}{2\alpha}+})$ with the structure of a $p$-matrix normed space. We will call this $p$-matrix normed space structure the \emph{$p$-matrix normed space structure~\eqref{OS2} with respect to $Y$} (this structure depends on $\alpha$ although not explicitly mentioned).

We want to point out that the norms on $H^{\infty}$ induced by the $p$-matrix normed space structures~\eqref{OS1} and~\eqref{OS2}, respectively, do not agree with the usual norms on these algebras as introduced in Section~\ref{sec:introduction}. We will always consider the spaces $H^{\infty}(\Sigma)$ with these non-standard norms.

\section{\texorpdfstring{$p$}{p}-Completely Bounded \texorpdfstring{$H^{\infty}$-}{Holomorphic Functional }Calculus}

In this section we show that the $H^{\infty}$-calculus of a (pseudo)-sectorial operator is $p$-completely bounded for the $p$-matrix normed space structures just defined. The proof will use vector-valued transference techniques and therefore we start with the case of a bounded group. Notice that a bounded group on a UMD-space always has a bounded $H^{\infty}(\Sigma_{\frac{\pi}{2}+})$-calculus and that in the next two propositions we consider $H_0^{\infty}(\Sigma_{\frac{\pi}{2}+})$ with the norm induced by the $p$-matrix normed space structure~\eqref{OS1}.

\begin{proposition}\label{prop:group_case}
	Let $(U(t))_{t \in \IR}$ be a bounded group on a subspace-quotient $SQ_X$ of a UMD-space $X$ with infinitesimal generator $-C$. Then the $H^{\infty}$-calculus homomorphism
		\[ u\colon H_0^{\infty}(\Sigma_{\frac{\pi}{2}+}) \to \mathcal{B}(SQ_X) \]
	is $p$-completely bounded for $p \in (1,\infty)$ and the $p$-matrix normed space structure~\eqref{OS1} with respect to $X$. 
\end{proposition}
	
\begin{proof}
	Let us first assume that the subspace-quotient is $X$ itself. Further let $M \coloneqq \sup_{t \in \IR} \norm{U(t)}$. One has for $[f_{ij}] \in M_n(H_0^{\infty}(\Sigma_{\frac{\pi}{2}+}))$ and the inverse Laplace transforms $[b_{ij}] \in M_n(L^1(\IR_+))$ determined by $f_{ij} = \hat{b}_{ij}$
		\begin{align*}
			\MoveEqLeft \norm{[f_{ij}(C)]}_{M_n(\mathcal{B}(X))} = \norm{[f_{ij}(C)]}_{\mathcal{B}(\ell_p^n(X))} = \biggnorm{\bigg[\int_0^{\infty} b_{ij}(t) U(t) \, dt\bigg]}_{\mathcal{B}(\ell_p^n(X))} \\
			& = \biggnorm{\int_0^{\infty} ([b_{ij}(t)] \otimes \Id) \, \mathcal{U}(t) \, dt}_{\mathcal{B}(\ell_p^n(X))},
		\end{align*}
	where $\mathcal{U}(t)$ is the diagonal operator $\diag(U(t), \ldots, U(t))$ in $\mathcal{B}(\ell_p^n(X))$. Since the matrix $[b_{ij}(t)] = [b_{ij}(t) \Id]$ is in the commutant of the group $(\mathcal{U}(t))_{t \in \IR}$, one has by the vector-valued transference principle~\cite[Theorem~4.1]{LeM99}
	 	\begin{align*}
			\MoveEqLeft \norm{[f_{ij}(C)]}_{M_n(\mathcal{B}(X))} \le M^2 \norm{\int_0^{\infty} V(t) \otimes [b_{ij}(t)] \otimes \Id \, dt}_{\mathcal{B}(L_p(\IR; \ell_p^n(X)))}. %
		\end{align*}
	Let $\mathcal{V}(t)$ be the diagonal operator $\diag(V(t), \ldots, V(t))$ on $\ell_p^n(L_p(\IR;X))$. Then after interchanging the order of the $L_p$-spaces we obtain
		\begin{align*}
			\MoveEqLeft \norm{[f_{ij}(C)]}_{M_n(\mathcal{B}(X))} \le M^2 \biggnorm{\int_0^{\infty} ([b_{ij}(t)] \otimes \Id) \, \mathcal{V}(t) \, dt}_{\mathcal{B}(\ell_p^n(L_p(\IR;X)))} \\
			& = M^2 \biggnorm{\bigg[ \int_0^{\infty} b_{ij}(t) V(t) \, dt \bigg]}_{\mathcal{B}(\ell_p^n(L_p(\IR;X)))} = M^2 \norm{[f_{ij}(B)]}_{M_n(L_p(\IR;X))},
		\end{align*}
		which is the $p$-complete boundedness of $u$. Now, if $SQ_X$ is a general subspace-quotient of $X$, one can repeat all the above arguments replacing $X$ by $SQ_X$. Further, note that the shift $(V(t))_{t \in \IR}$ in $L_p(\IR;X)$ naturally restricts to the shift $(V_{SQ_X}(t))_{t \in \IR}$ on $L_p(\IR;SQ_X)$ with generator $B_{SQ_X}$. From this and the definition of the functional calculus it then follows immediately that for all $[f_{ij}] \in M_n(H_0^{\infty}(\Sigma_{\frac{\pi}{2}+}))$
			\begin{align*} 
				\norm{[f_{ij}(C)]}_{M_n(\mathcal{B}(SQ_X))} & \le M^2 \norm{[f_{ij}(B_{SQ_X})]}_{M_n(L_p(\IR;SQ_X))} \\
				& \le M^2 \norm{[f_{ij}(B)]}_{M_n(L_p(\IR;X))}. \qedhere
			\end{align*}
\end{proof}

Let now $(T(t))_{t \ge 0}$ be a bounded analytic $C_0$-semigroup with sectorial generator $-A$ on a subspace-quotient $SQ_X$ of a UMD-space $X$. If $A$ has a bounded $H^{\infty}$-calculus with $\omega_{H^{\infty}}(A) < \frac{\pi}{2}$, then the Fröhlich--Weis theorem~\cite[Corollary~5.4]{FroWei06} yields that $(T(t))_{t \ge 0}$ dilates to a bounded $C_0$-group $(U(t))_{t \in \IR}$ with negative generator $-C$ on the UMD-space $SQ_Y \coloneqq L_p([0,1];SQ_X)$ which is itself a subspace-quotient of $Y \coloneqq L_p([0,1];X)$ (one can choose $1 < p < \infty$), that is
	\[ JT(t) = PU(t)J, \]
where $J\colon SQ_X \to SQ_Y$ is an isometric embedding and $P\colon SQ_Y \to SQ_Y$ is a bounded projection onto $\Im(J)$. Then one has for $M \coloneqq \sup_{t \in \IR} \norm{U(t)}$ by Proposition~\ref{prop:group_case}
	\begin{align*}
		\norm{[f_{ij}(A)]}_{M_n(\mathcal{B}(SQ_X))} & = \norm{[J f_{ij}(A)]}_{M_n(\mathcal{B}(SQ_X,SQ_Y))} \\  
		& = \norm{[Pf_{ij}(C)J]}_{M_n(\mathcal{B}(SQ_X,SQ_Y))} \le \norm{P} \norm{[f_{ij}(C)]}_{M_n(\mathcal{B}(SQ_Y))} \\
		& \le \norm{P} M^2 \norm{f_{ij}(B)}_{M_n(L_p(\IR; Y))},
	\end{align*}
where $-B$ is the generator of the shift group on $L_p(\IR;Y)$. We have shown the following proposition.

\begin{proposition}\label{prop:p-completely-bound-os1}
	Let $A$ be a sectorial operator on a subspace-quotient $SQ_X$ of a UMD-space $X$ with a bounded $H^{\infty}(\Sigma_{\theta})$-calculus for some $\theta < \frac{\pi}{2}$. Then the $H^{\infty}$-calculus homomorphism
		\[ u\colon H_0^{\infty}(\Sigma_{\frac{\pi}{2}+}) \to \mathcal{B}(SQ_X) \]
	is $p$-completely bounded for $p \in (1,\infty)$ and the $p$-matrix normed space structure~\eqref{OS1} with respect to $L_p([0,1];X)$.
\end{proposition}

Using the $p$-matrix normed space structure~\eqref{OS2} instead, we obtain our main result of this section. In what follows, we consider $H_0^{\infty}(\Sigma_{\frac{\pi}{2\alpha}+})$ with the normed induced by~\eqref{OS2}.

\begin{theorem}\label{thm:p-completely-bound-os2-pseudo}
	Let $A$ be a sectorial operator on a subspace-quotient $SQ_X$ of a UMD-space $X$ with a bounded $H^{\infty}(\Sigma_{\theta})$-calculus for some $\theta < \frac{\pi}{2}$. Then for $\alpha \in [1, \frac{\pi}{2\theta})$ the $H^{\infty}$-calculus homomorphism
		\[ u\colon H_0^{\infty}(\Sigma_{\frac{\pi}{2\alpha}+}) \to \mathcal{B}(SQ_X) \]
	is $p$-completely bounded for $p \in (1,\infty)$ and the $p$-matrix normed space structure~\eqref{OS2} with respect to $L_p([0,1];X)$.
\end{theorem}

\begin{proof}
	Let $\alpha \in [1, \frac{\pi}{2\theta})$. Then $A^{\alpha}$ has a bounded $H^{\infty}(\Sigma_{\alpha \theta})$-calculus with $\alpha \theta < \frac{\pi}{2}$ by Proposition~\ref{prop:fractional_powers}. We apply Proposition~\ref{prop:p-completely-bound-os1} to $A^{\alpha}$ and obtain that for $C = C(\alpha) > 0$
		\[ \norm{[\tilde{f}_{ij}(A^{\alpha})]} \le C \norm{[\tilde{f}_{ij}(B)]} \]
	for all $[\tilde{f}_{ij}] \in M_n(H_0^{\infty}(\Sigma_{\frac{\pi}{2}+}))$. Now let $[f_{ij}] \in M_n(H_0^{\infty}(\Sigma_{\frac{\pi}{2\alpha}+}))$. Then there exist $\tilde{f}_{ij} \in H_0^{\infty}(\Sigma_{\frac{\pi}{2}+})$ such that $\tilde{f}_{ij}(\lambda^{\alpha}) = f_{ij}(\lambda)$. Now,
		\begin{align*}
			\norm{[f_{ij}(A)]} = \norm{[(\tilde{f}_{ij} \circ \cdot ^{\alpha})(A)]} = \norm{[\tilde{f}_{ij}(A^{\alpha})]} \le C \norm{[\tilde{f}_{ij}(B)]} = C \norm{[f_{ij}(B^{\frac{1}{\alpha}})]}. \quad \qedhere
		\end{align*} 
\end{proof}

The $p$-complete boundedness of the  functional calculus homomorphism extends from $H_0^{\infty}(\Sigma_{\frac{\pi}{2\alpha}+})$ to $H^{\infty}(\Sigma_{\frac{\pi}{2\alpha}+})$.

\begin{corollary}\label{thm:p-completely-bound-os2}
	Let $A$ be a sectorial operator on a subspace-quotient $SQ_X$ of a UMD-space $X$ with a bounded $H^{\infty}(\Sigma_{\theta})$-calculus for some $\theta < \frac{\pi}{2}$. Then for $\alpha \in [1, \frac{\pi}{2\theta})$ the $H^{\infty}$-calculus homomorphism
		\[ u\colon H^{\infty}(\Sigma_{\frac{\pi}{2\alpha}+}) \to \mathcal{B}(SQ_X) \]
	is $p$-completely bounded for $p \in (1,\infty)$ and the $p$-matrix normed space structure~\eqref{OS2} with respect to $L_p([0,1];X)$.

\end{corollary}

\begin{proof}
	Let $Y \coloneqq L_p([0,1];X)$. The case of general $[f_{ij}] \in M_n(H^{\infty}(\Sigma_{\frac{\pi}{2\alpha}+}))$ follows from Theorem~\ref{thm:p-completely-bound-os2-pseudo} by the following approximation argument. One has $[\rho_k f_{ij}] \in M_n(H_0^{\infty}(\Sigma_{\frac{\pi}{2\alpha}+}))$ for all $k \in \IN$ and therefore obtains for some constant $C \ge 0$ that for all $n \in \IN$ and all $x_1, \ldots, x_n \in X$ 
		\begin{align*}
			\MoveEqLeft \sum_{i=1}^n \biggnorm{\sum_{j=1}^n f_{ij}(A) x_j}^p = \lim_{k \to \infty} \sum_{i=1}^n \biggnorm{\sum_{j=1}^n (\rho_k f_{ij})(A)x_j }^p \\
			& \le C^{p} \sum_{i=1}^n \norm{x_i}^p \liminf_{k \to \infty} \norm{[\rho_k(B^{\frac{1}{\alpha}}) f_{ij}(B^{\frac{1}{\alpha}}) ]}^p_{M_n(\mathcal{B}(L_p(\IR;Y)))} \\
			& \le C^{p} \sup_{k \in \IN} \norm{\rho_k (B^{\frac{1}{\alpha}} )}^p \norm{[f_{ij}(B^{\frac{1}{\alpha}})]}_{M_n(\mathcal{B}(L_p(\IR;Y)))}^p \sum_{i=1}^n \norm{x_i}^p.
		\end{align*}
	Since $\sup_{k \in \IN} \norm{\rho_k(B^{1/\alpha})} < \infty$, we showed $\norm{[f_{ij}(A)]} \le \tilde{C} \norm{[f_{ij}(B^{1/\alpha})]}$.
\end{proof}

\section{The Structure Theorem}\label{sec:structure_theorem}

We now apply Pisier's factorization theorem for $p$-completely bounded maps to the homomorphism obtained from the bounded $H^{\infty}$-calculus. We start with some terminology.

Let $X$ be a Banach space and $Y = L_p([0,1]; X)$ (for some $p \in (1, \infty)$). Further let $(\Omega_j, \mu_j)_{j \in J}$ be a family of measure spaces and $\mathcal{U}$ be an ultrafilter on $J$ and $\alpha \in [1, \infty)$. Then for each $j \in J$ and $f \in H^{\infty}(\Sigma_{\frac{\pi}{2\alpha}+})$ we have canonical maps
	\begin{align*}
		\pi_j(f)\colon L_p(\Omega_j; L_p(\IR; Y)) & \to L_p(\Omega_j; L_p(\IR;Y)) \\
		g & \mapsto [\omega \mapsto f(B^{1/\alpha})g(\omega)]. 
	\end{align*}
These mappings induce for all $f \in H^{\infty}(\Sigma_{\frac{\pi}{2\alpha}+})$ a map $\pi(f) \in \mathcal{B}(\hat{X})$ in the ultraproduct $\hat{X} \coloneqq \prod_{j \in J }L_p(\Omega_j; L_p(\IR; Y))/\mathcal{U}$ with $\norm{\pi(f)}\le \sup_{j \in J} \norm{\pi_j(f)} \le \norm{f(B^{1/\alpha})} \le C_{\alpha} \norm{f}_{H^{\infty}}$ by the boundedness of the $H^{\infty}$-calculus. More generally, each $T \in \mathcal{B}(L_p(\IR;Y))$ induces a map on the above ultraproduct with operator norm at most $\norm{T}$. Note that this in particular implies that a holomorphic mapping $z \mapsto T_z$ in $\mathcal{B}(L_p(\IR;Y))$ induces a holomorphic mapping in this ultraproduct. 

We now formulate a special case of Pisier's factorization theorem.

\begin{theorem}[Pisier's factorization theorem for completely bounded maps]\label{thm:pisier} Let $\mathcal{A} \subset \mathcal{B}(Z)$ for a Banach space $Z$ be a unital subalgebra and $u\colon \mathcal{A} \to \mathcal{B}(X)$ be a $p$-completely bounded unital algebra homomorphism for some $p \in (1,\infty)$. Then there exists a family of measure spaces $(\Omega_j, \mu_j)_{j \in J}$ and an ultrafilter $\mathcal{U}$ on $J$ such that for the ultraproduct $\hat{X} \coloneqq \prod_{j \in J} L_p(\Omega_j; Z) / \mathcal{U}$ there are closed subspaces $N \subset M \subset \hat{X}$ and an isomorphism $S\colon X \to M/N$ such that for $a \in \mathcal{A}$ the operators $\pi(a)$ defined as the ultraproducts $\prod_{j \in J} \pi_j(a) / \mathcal{U}$, where
	\[ \pi_j(a)\colon L_p(\Omega_j;Z) \to L_p(\Omega_j;Z), \qquad (\pi_j(a)f)(\omega) \coloneqq a(f(\omega)) \]
satisfy $\pi(a)M \subset M$ and $\pi(a)N \subset N$ for all $a \in \mathcal{A}$ and the induced mappings $\hat{\pi}(a)\colon M/N \to M/N$   satisfy
	\[ u(a) = S^{-1} \hat{\pi}(a) S \qquad \text{for all } a \in \mathcal{A}. \]
\end{theorem}

This follows from~\cite[Theorem~3.2]{Pis90}. There the theorem is first proved for a map $u\colon \mathcal{A} \to \mathcal{B}(\ell_1(\Gamma), \ell_{\infty}(\Gamma'))$ for some sets $\Gamma$ and $\Gamma'$, where one can choose $M = \hat{X}$ and $N = 0$. In the general case of Theorem~\ref{thm:pisier} one then chooses a metric surjection $Q\colon \ell_1(\Gamma) \to X$ (i.e. $Q^*$ is an isometric embedding) and an isometric embedding $J\colon X  \to \ell_{\infty}(\Gamma')$ and deduces the theorem from the special case above applied to $a \mapsto Ju(a)Q$.

In our concrete case of the functional calculus, we get the following.

\begin{theorem}\label{thm:pisier_applied_to_functional_calculus} Let $A$ be a sectorial operator on a subspace-quotient $SQ_X$ of a UMD-space $X$ with a bounded $H^{\infty}(\Sigma_{\theta})$-calculus for some $\theta < \frac{\pi}{2}$. Then for each $\psi \in (\theta, \frac{\pi}{2})$ there exist $(\Omega_j, \mu_j)_{j \in J}$, $\mathcal{U}$ and $\pi$ as above together with subspaces $N \subset M$ of the ultraproduct
	\[ \prod_{j \in J }L_p(\Omega_j; L_p(\IR; L_p([0,1];X)))/\mathcal{U} \] 
	such that
	\begin{equation*} 
		\pi(f)M \subset M \quad \text{and} \quad \pi(f)N \subset N \qquad \text{for all } f \in H^{\infty}(\Sigma_{\psi}).
	\end{equation*}
	Moreover, if $\hat{\pi}(f)\colon M / N \to M / N$ denotes the induced mappings, there exists an isomorphism $S\colon SQ_X \to M / N$ such that
	\[ u(f) = S^{-1} \hat{\pi}(f) S \qquad \text{for all } f \in H^{\infty}(\Sigma_{\psi}). \]
\end{theorem}

\begin{proof}
	Let $\psi > \theta$. We have shown in Theorem~\ref{thm:p-completely-bound-os2} that for a suitably chosen $\alpha > 1$ the functional calculus homomorphism $u\colon H^{\infty}(\Sigma_{\psi}) \to \mathcal{B}(SQ_X)$ is $p$-completely bounded for the $p$-matrix normed space structure~\eqref{OS2} with respect to $L_p([0,1];X)$. We now apply Theorem~\ref{thm:pisier} to $u$.
\end{proof}

\section{Properties of the Semigroup}

In particular in Theorem~\ref{thm:pisier_applied_to_functional_calculus} one has, for $\psi < \frac{\pi}{2}$, that $e_z \in H^{\infty}(\Sigma_{\psi})$ for $z \in \Sigma_{\pi/2 - \psi}$ and one obtains $T(z) = u(e_z) = S^{-1} \hat{\pi}(e_z) S$, where $(T(z))_{z \in \Sigma_{\pi/2 - \omega(A)}}$ is the analytic semigroup generated by $-A$. 

Therefore we are interested in the properties of the semigroup $(\pi(e_z))_{z \in \Sigma_{\frac{\pi}{2}-\psi}}$ which in turn leads us to the investigation of the semigroup generated by $-B^{1/\alpha}$. Here one has the following general result for fractional powers of semigroup generators.

\begin{theorem}\label{thm:fractional_positive} Let $-A$ be the generator of a $C_0$-semigroup $(T(t))_{t \ge 0}$ on some Banach space $X$. Then for $\alpha \in (0,1)$ the bounded analytic $C_0$-semigroup\\ $(T_{\alpha}(z))_{z \in \Sigma_{\frac{\pi}{2} - \omega(A^{\alpha})}}$ generated by $-A^{\alpha}$ has the following properties:
	\begin{itemize}
		\item one has $\norm{T_{\alpha}(t)} \le 1$ for all $t \ge 0$ if $\norm{T(t)} \le 1$ for all $t \ge 0$.
	\end{itemize}
	Moreover, if $X$ is a Banach lattice, then one has
	\begin{itemize}
		\item $T_{\alpha}(t) \ge 0$ if $T(t) \ge 0$.
	\end{itemize}
\end{theorem}
\begin{proof}
	The assertions follow from the explicit representation of the semigroup $(T_{\alpha}(t))_{t \ge 0}$ (see~\cite[IX, 11]{Yos80})
		\[ T_{\alpha}(t)x = \int_0^{\infty} f_{t, \alpha}(s) T(s)x \, ds \qquad \text{for all } t > 0 \text{ and } x \in X, \]
	where $f_{t, \alpha}$ is a function with $f_{t, \alpha} \ge 0$ and $\int_0^{\infty} f_{t, \alpha}(s) \, ds = 1$.
\end{proof}

In our case we apply the above theorem to the generators $-B$ of the contractive shift semigroups. Moreover, if $X$ and therefore $L_p([0,1];X)$ are Banach lattices, then the shift semigroup is positive with respect to the natural Banach lattice structure on $L_p(\IR; L_p([0,1];X))$. Hence, $(\pi(e_t))_{t \ge 0}$ is a (positive if $X$ is a Banach lattice) contractive semigroup. Here we used the fact that the ultraproduct of positive operators is positive on ultraproducts of Banach lattices.

\begin{remark}\label{rem:shift_power_contractive_sector}
	One can now ask if the analytic semigroup generated by $-B^{\alpha}$ is contractive on the whole sector where the semigroup is defined. Note that the Fourier transform diagonalizes $B$ to the multiplication operator with $ix$. The semigroup $(T_{\alpha}(z))$ generated by $-B^{\alpha}$ is then given by the Fourier multipliers $\exp(-z (ix)^{\alpha})$ for $\abs{\arg z} < \frac{\pi}{2} (1-\alpha) \eqqcolon \phi$. Now, if $X$ is a Hilbert space, then Plancherel's theorem yields that $(T_{\alpha}(z))$ is a sectorial contractive semigroup on $L_2(\IR;X)$. However, even in the case $X = \IC$ one does not have contractivity on the whole sector for $p \in (1, \infty) \setminus \{2\}$. Indeed, $\norm{T(z)}$ is even unbounded on $\Sigma_{\phi} \cap \{ z: \abs{z} \le 1\}$. For this, by~\cite[Proposition 3.9.1]{ABHN11}, it suffices to show that $-e^{i\phi} A^{\alpha}$ does not generate a $C_0$-semigroup on $L_p$.
	
	Assume that this would be the case. Then in particular $\exp(-e^{i\phi} (ix)^{\alpha})$ is a bounded Fourier multiplier on $L_p$. Using the boundedness of the Hilbert transform for $p \in (1, \infty)$ one sees that then
		\[ m(\xi) \coloneqq \begin{cases} e^{-i \xi^{\alpha}} & ( \xi \ge 2), \\ 0 & (\xi < 2) \end{cases} \]
	is also a bounded Fourier multiplier. However, this contradicts~\cite[Theorem~E.4b) (i)]{ABHN11}.
\end{remark}

\begin{proposition}\label{prop:first_factorization} Let $A$ be a pseudo-sectorial operator on a subspace-quotient $SQ_X$ of a UMD-Banach lattice $X$ with a bounded $H^{\infty}(\Sigma_{\theta})$-calculus for some $\theta < \frac{\pi}{2}$ and let $(T(z))$ be the bounded analytic $C_0$-semigroup generated by $-A$. Then for each $\psi \in (\theta, \frac{\pi}{2})$ there exists a UMD-Banach lattice $\hat{X}$ and a bounded analytic semigroup $(\Pi(z))_{z \in \Sigma_{\pi/2 - \psi}}$ on $\hat{X}$, positive and contractive on the real line, together with subspaces $N \subset M \subset \hat{X}$ which are invariant under $(\Pi(z))$ and an isomorphism $S\colon SQ_X \to M/N$ such that the induced semigroup $(\hat{\Pi}(z))_{z \in \Sigma_{\pi/2 - \psi}}$ on $M/N$ satisfies
	\[ T(z) = S^{-1} \hat{\Pi}(z) S \qquad \text{for all } z \in \Sigma_{\pi/2 - \psi}. \]
Moreover, if $X$ is separable, then $\hat{X}$ can be chosen separable as well.
\end{proposition}

\begin{proof} Let $\psi \in (\theta, \frac{\pi}{2})$. We first assume that $A$ is sectorial. For $z \in \Sigma_{\frac{\pi}{2}- \psi}$ we set $\Pi(z) \coloneqq \pi(e_z)$, where $\pi$ is obtained from Theorem~\ref{thm:pisier_applied_to_functional_calculus}. We denote the ultraproduct constructed there by $\hat{X}$. Note that the analyticity of the semigroup $(\Pi(z))$ follows from the remarks at the beginning of Section~\ref{sec:structure_theorem}. We now verify that $\hat{X}$ is a UMD-space. Notice that the UMD-property is a super-property, that is a property which is inherited by spaces which are finitely representable in a space having this property (for an introduction to these notions we refer to~\cite[Chapter~3]{Pis11b}). Let again be $Y = L_p([0,1];X)$ for $p \in (1,\infty)$. Now, $L_p(\Omega_j; L_p(\IR;Y))$ is finitely representable in $\ell_p(L_p(\IR;Y))$ (this can be proved as in~\cite[Theorem~6.2]{FHH+11}) and therefore $\hat{X} =\prod_{j \in J} L_p(\Omega_j;L_p(\IR;Y)) / \mathcal{U}$ is finitely representable in (see $\ell_p(L_p(\IR;Y))$~\cite[Lemma~3.4.8]{Pis11b}). This shows that $\hat{X}$ is a UMD-space because $X$ and therefore $\ell_p(L_p(\IR;Y))$ are UMD-spaces.

Now assume that $SQ_X$ is separable. Then there exist a metric surjection $Q\colon \ell_1 \to SQ_X$ and an isometric embedding $J\colon SQ_X \to \ell_{\infty}$. By the short description of the proof of Pisier's factorization theorem given above, the ultraproduct $\hat{X}$ is constructed via a reduction to the completely bounded map $\ell_1 \xrightarrow{Q} SQ_X \xrightarrow{u(f)} SQ_X \xrightarrow{J} \ell_{\infty}$ which factorizes as
	\[ \ell_1 \xrightarrow{v_1} \hat{X} \xrightarrow{\pi(f)} \hat{X} \xrightarrow{v_2} \ell_{\infty} \]
for two maps $v_1\colon \ell_1 \to \hat{X}$ and $v_2\colon \hat{X} \to \ell_{\infty}$. We can now replace $\hat{X}$ by a separable closed vector sublattice constructed as follows: Let $X_0$ be the closed vector sublattice generated by $v_1(\ell_1)$. Then for $n \ge 1$ define $X_n$ inductively as the closed vector sublattice generated by elements of the form $\Pi(q)X_{n-1}$, where $q$ runs through $(\IQ + i\IQ) \cap \Sigma_{\frac{\pi}{2} - \psi}$. Then for all $n \in \IN_{0}$ the lattice $X_n$ is separable and by the continuity of the semigroup one has $\Pi(z)X_n \subset X_{n+1}$ for all $z \in \Sigma_{\frac{\pi}{2} - \psi}$. Hence, the closure of $\hat{Y} \coloneqq \cup_{n \in \IN_{0}} X_n$ is a closed separable vector sublattice of $\hat{X}$ which is invariant under the semigroup $(\Pi(z))$. We may now finish the proof with $\hat{X}$ replaced by $\hat{Y}$.

Now let $A$ be pseudo-sectorial. Then one can decompose $SQ_X$ into $SQ_X = \overline{R(A)} \oplus N(A)$ such that $A$ is of the form $A = \begin{pmatrix} A_{00} & 0 \\ 0 & 0 \end{pmatrix}$. Then $\overline{R(A)}$ is a subspace-quotient of $X$ as well and therefore, by Corollary~\ref{thm:p-completely-bound-os2}, the sectorial operator $A_{00}$ on $\overline{R(A)}$ has a $p$-completely bounded $H^{\infty}(\Sigma_{\psi})$-calculus for $p \in (1, \infty)$ and the $p$-matrix normed space structure \eqref{OS2} with respect to $L_p([0,1];X)$. Hence, by Theorem~\ref{thm:pisier_applied_to_functional_calculus} and the first part of the proof there exists a UMD-Banach lattice $\hat{X}$ (which is separable if $SQ_X$ is separable), a semigroup $(\Pi(z))_{z \in \Sigma_{\frac{\pi}{2} - \psi}}$ on $\hat{X}$, $(\Pi(z))$-invariant subspaces $N \subset M \subset \hat{X}$ and an isomorphism $S\colon \overline{R(A)} \to M/N$ such that the induced semigroup $(\hat{\Pi}(z))_{z \in \Sigma_{\frac{\pi}{2} - \psi}}$ on $M / N$ satisfies
 	\[ T_{|\overline{R(A)}}(z) = S^{-1} \hat{\Pi}(z) S \qquad \text{for all } z \in \Sigma_{\frac{\pi}{2} - \psi}. \]
Note that with respect to the decomposition $X = \overline{R(A)} \oplus N(A)$ one has $T(z) = T_{|\overline{R(A)}}(z) \oplus \Id$. Let now $P$ be the projection onto $N(A)$ and let $\hat{Z} = \hat{X} \oplus X$ be the direct sum with its natural Banach lattice structure. Clearly, $\hat{Z}$ is UMD and separable if $X$ is separable. Now define $V_1\colon SQ_X \to M/N \oplus SQ_X$ as the matrix $V_1 = \begin{pmatrix} S & 0 \\ 0 & \iota \end{pmatrix}$, where $\iota$ is the inclusion of $N(A)$ into $SQ_X$, and $V_2\colon M/N \oplus SQ_X \to SQ_X$ as the matrix $V_2 = \begin{pmatrix} S^{-1} & 0 \\ 0 & P \end{pmatrix}$. Let $\tilde{\pi}(f)$ be the extension of $\pi$ to $\hat{Z}$ by the identity on the second component for all $f \in H^{\infty}(\Sigma_{\psi})$. If we also extend the functional calculus homomorphism $u\colon H^{\infty}(\Sigma_{\psi}) \to \mathcal{B}(\overline{R(A)})$ to the homomorphism $\tilde{u}\colon H^{\infty}(\Sigma_{\psi}) \to \mathcal{B}(X)$ by the identity on the second component, then we have for the compression $\hat{\tilde{\pi}}$ of $\tilde{\pi}$ to $M/N \oplus SQ_X$
	\[ \tilde{u}(f) = V_2 \hat{\tilde{\pi}}(f) V_1 \qquad \text{for all } f \in H^{\infty}(\Sigma_{\psi}). \]
In particular, note that $\tilde{u}$ and $\tilde{\pi}$ are constructed in a way such that one has
	\[ T(z) = \tilde{u}(e_z) = V_2 \hat{\tilde{\pi}}(e_z) V_1 \qquad \text{for all } z \in  \Sigma_{\frac{\pi}{2} - \psi}. \]
We can now apply \cite[Proposition~4.2]{Pis01} which shows that there exist subspaces $E_2 \subset E_1 \subset M/N \oplus SQ_X$ being invariant under $\hat{\tilde{\pi}}(f)$ for all $f \in H^{\infty}(\Sigma_{\psi})$ and an isomorphism $\hat{S}\colon SQ_X \to E_1 / E_2$ such that the compression $\pi_C$ of $\hat{\tilde{\pi}}$ satisfies
	\[ \tilde{u}(f) = \hat{S}^{-1} \pi_C(f) \hat{S} \qquad \text{for all } f \in H^{\infty}(\Sigma_{\psi}). \]
Note that $E_1$ and $E_2$ can be seen as subspaces of $\hat{Z}$. Then the assertion follows by inserting $f = e_z$ for $z \in \Sigma_{\frac{\pi}{2}- \psi}$.
\end{proof}

\begin{remark} Notice that in the special case where $X$ is an $L_p$-space the semigroup $(\Pi(z))$ lives on (a closed vector sublattice of) an ultraproduct of $L_p$-spaces which, by Kakutani's theorem~\cite[Theorem 1.b.2]{LinTza79}, is order isometric to an $L_p$-space. In this case one can therefore realize $\hat{X}$ as an $L_p$-space as well.
\end{remark}

\section{Strong Continuity}

Note that there is one drawback of the ultraproduct construction used above. In general the obtained semigroup $(\pi(e_z))_{z \in \Sigma}$ is not strongly continuous. However, the semigroup is a \emph{bounded analytic semigroup}, that is a bounded holomorphic mapping from the sector $\Sigma$ to the Banach space $\hat{X}$ satisfying the semigroup law. Since the space $\hat{X}$ is reflexive one knows that $P \coloneqq \lim_{t \downarrow 0} \pi(e_t)$ exists strongly and is a projection onto $X_1 \coloneqq P\hat{X}$~\cite[beginning of §5]{Are01}. Then the restriction $\pi(e_z)_{|X_1}$ defines a strongly continuous semigroup on $X_1$, whereas $\pi(e_z)_{|X_0} = 0$, where $X_0 \coloneqq (\Id - P)(\hat{X})$. In the next theorem we show that we can always reduce to this strongly continuous part.

\begin{theorem}\label{thm:generic} Let $-A$ be the generator of a bounded analytic $C_0$-semigroup $(T(z))$ on a subspace-quotient $SQ_X$ of a UMD-Banach lattice $X$ such hat $\theta \coloneqq \omega_{H^{\infty}}(A) < \frac{\pi}{2}$. Then for each $\psi \in (\theta, \frac{\pi}{2})$ there exists a UMD-Banach lattice $X_1$ and a bounded analytic $C_0$-semigroup $(R(z))_{z \in \Sigma_{\frac{\pi}{2} - \psi}}$ on $X_1$ which is positive and contractive on the real line together with subspaces $N \subset M \subset X_1$ which are invariant under $(R(z))$ and an isomorphism $S\colon SQ_X \to M/N$ such that the induced semigroup $(\hat{R}(z))_{z \in \Sigma_{\frac{\pi}{2} - \psi}}$ on $M/N$ satisfies
	\[ T(z) = S^{-1} \hat{R}(z) S \qquad \text{for all } z \in \Sigma_{\pi/2 - \psi}. \]
Moreover, if $X$ is separable, then $X_1$ can be chosen separable as well.
\end{theorem}

\begin{proof} Let $P \coloneqq \lim_{t \downarrow 0} \pi(e_t)$ in the strong sense be the projection onto $X_1$ as above. It follows directly from the definition that $P$ is a positive contractive projection. Hence, $X_1$ is a lattice subspace and therefore itself a Banach lattice with the induced order structure~\cite[Theorem~5.59]{AbrAli02}. Since the UMD-property passes to subspaces, $X_1$ is a UMD-Banach lattice. Let $(R(z))$ be the restriction of $(\Pi(z))$ to $X_1$ given by Proposition~\ref{prop:first_factorization}. By construction $(R(z))$ is a bounded analytic $C_0$-semigroup on $X_1$ and positive and contractive on the real line. 
Let $\mathcal{A}$ be the closed unital subalgebra of $H^{\infty}(\Sigma_{\psi})$ generated by $e_z$, that is $\mathcal{A} = \overline{\linspan}\{e_z: z \in \Sigma_{\psi} \cup \{0\} \}$. It follows from the discussion above and the continuity of $f \mapsto \pi(f)$ that $\pi(\mathcal{A})X_1 \subset X_1$ (we replace $\pi$ by its extension constructed in Proposition~\ref{prop:first_factorization} if $A$ is merely pseudo-sectorial). Furthermore, note that $\pi(\mathcal{A})$ leaves $PM$ and $PN$ invariant and that $P$ restricts to a projection on $M$ and $N$ by the invariance under the semigroup. This allows us to define maps 
	\[ V_1\colon SQ_X \xrightarrow[S]{} M/N \xrightarrow[P]{} PM/PN, \quad V_2\colon PM/PN \hookrightarrow M/PN \to M/N \xrightarrow[S^{-1}]{} SQ_X. \]
	As $\pi(\mathcal{A})$ acts only non-trivially on $X_1$ one obtains
	\[ u(a) = V_2 \hat{\pi}(a) V_1 \qquad \text{for all } a \in \mathcal{A}, \]
	where the hat indicates the mapping $PM/PN \to PM/PN$ induced by $\pi(a)_{|X_1}$. We now apply~\cite[Proposition~4.2]{Pis01} which shows that there exist $\hat{\pi}$-invariant subspaces $E_2 \subset E_1 \subset PM/PN$ and an isomorphism $\tilde{S}\colon SQ_X \xrightarrow{\sim} E_1/E_2$ such that the compression $\tilde{\pi}$ of $\hat{\pi}$ to $E_1/E_2$ satisfies
	\[ u(a) = \tilde{S}^{-1} \tilde{\pi}(a) \tilde{S} \qquad \text{for all } a \in \mathcal{A}. \]
Note that $E_1/E_2$ is a subspace-quotient of $X_1$ and that one has in particular
	\[ T(z) = \tilde{S}^{-1} \tilde{\pi}(a) \tilde{S} \qquad \text{for all } a \in \mathcal{A}. \qedhere \]
\end{proof}

\begin{remark}
	Note that the above arguments also work for general UMD-Banach spaces instead of UMD-Banach lattices. The underlying Banach space $X_1$ of the constructed semigroup $(R(z))$ inherits properties from the Banach space $X$ if they are stable under the constructed ultraproduct as well as under quotients and subspaces. For example, if $X$ is uniformly convex, then the ultraproduct $\hat{X}$ (for $p = 2$) is uniformly convex as well (with the same modulus of uniform convexity) because it is finitely representable in the uniformly convex space $\ell_p(L_p(\IR;Y))$ (see~\cite[Corollary~1]{Day41b}). As uniform convexity also passes to subspaces and quotients, this shows that if $-A$ generates a bounded analytic semigroup $(T(z))_{z \in \Sigma}$ on $X$ with $\omega_{H^{\infty}}(A) < \frac{\pi}{2}$, then there exists an isomorphism $S\colon X \to Z$ to a uniformly convex Banach space $Z$ such that $\norm{ST(t)S^{-1}}_{\mathcal{B}(Z)} \le 1$ for all $t \ge 0$.
\end{remark}

\begin{remark}
	Note that if the semigroup is defined on a Hilbert space, then the used ultraproduct construction also yields a Hilbert space. The subspace quotient $M/N$ in Theorem~\ref{thm:generic} then is a Hilbert space as well. So in this case $S$ is an isomorphism between two Hilbert spaces. Moreover, the constructed semigroup is even contractive on the whole sector by Remark~\ref{rem:shift_power_contractive_sector}. Therefore our general construction recovers Le Merdy's result in the Hilbert space case.
\end{remark}

\begin{remark}
	The assumption $\omega_{H^{\infty}}(A) < \frac{\pi}{2}$ is crucial. In \cite[Proposition 4.8]{Mer98}, Le Merdy gives an example of a bounded $C_0$-semigroup with an invertible generator $-A$ satisfying $\omega_{H^{\infty}}(A) = \frac{\pi}{2}$ on a Hilbert space which is not similar to a contraction.
	
	In the same spirit, a classical counterexample by P.R. Chernoff~\cite{Che76} shows that there exists a semigroup on some Hilbert space with a bounded generator which is bounded on the real line but not similar to a contractive semigroup.
\end{remark}

\begin{corollary}\label{cor:generic_lp} Let $-A$ be the generator of a bounded analytic $C_0$-semigroup $(T(z))$ on a subspace quotient $SQ_{L_p}$ of some $L_p$-space $L_p(\Omega)$ for $p \in (1,\infty)$ such that $\theta \coloneqq \omega_{H^{\infty}}(A) < \frac{\pi}{2}$. Then for each $\psi \in (\theta, \frac{\pi}{2})$ there exists an $L_p$-space $L_p(\tilde{\Omega})$ and a bounded analytic $C_0$-semigroup $(R(z))_{z \in \Sigma_{\frac{\pi}{2} - \psi}}$ on $L_p(\tilde{\Omega})$ which is contractive and positive on the real line together with subspaces $N \subset M \subset L_p(\tilde{\Omega})$ which are invariant under $(R(z))$, and an isomorphism $S\colon SQ_{L_p} \to M/N$ such that the induced semigroup $(\hat{R}(z))_{z \in \Sigma_{\frac{\pi}{2} - \psi}}$ on $M/N$ satisfies
	\[ T(z) = S^{-1} \hat{R}(z) S \qquad \text{for all } z \in \Sigma_{\pi/2 - \psi}. \]
Moreover, if $L_p(\Omega)$ is separable, then $L_p(\tilde{\Omega})$ can be chosen separable as well.
\end{corollary}
\begin{proof}
	Note that in this special case $P$ is a positive and contractive projection defined on some $L_p$-space. It is known that in this case $\Im P$ is a closed vector sublattice~\cite[Problem~5.3.12]{AbrAli02b} of an $L_p$-space and therefore order isometric to an $L_p$-space by Kakutani's theorem.
\end{proof}

\begin{remark}
	We do not know whether in the case $SQ_{L_p} = L_p(\Omega)$ the above corollary remains true without passing to a subspace-quotient, i.e. whether one always has
		\[ T(z) = S^{-1} R(z) S \qquad \text{for all } z \in \Sigma_{\pi/2 - \psi} \]
	for a bounded analytic $C_0$-semigroup $(R(z))$ one some $L_p$-space which is positive and contractive on the real line.
\end{remark}

In the case where the semigroup is defined on a subspace-quotient of an $L_p$-space one obtains a particularly nice equivalence.

\begin{corollary}
	Let $-A$ be the generator of a bounded analytic $C_0$-semigroup $(T(z))$ on a subspace-quotient $SQ_{L_p}$ of an $L_p$-space $L_p(\Omega)$. Then the following are equivalent.
	
	\begin{equiv_enum}
		\item\label{char_hinfty_i} The sectorial operator A has a bounded $H^{\infty}$-calculus with $\omega_{H^{\infty}}(A) < \frac{\pi}{2}$.
		\item\label{char_hinfty_ii} There exists a bounded analytic $C_0$-semigroup $(R(z))_{z \in \Sigma}$ on an $L_p$-space $L_p(\tilde{\Omega})$ which is contractive and positive on the real line together with $(R(z))$-invariant subspaces $N,M$ with $N \subset M$ and an isomorphism $S\colon SQ_{L_p} \to M/N$ such that the induced semigroup $(\hat{R}(z))_{z \in \Sigma}$ on $M/N$ satisfies
			\[ T(z) = S^{-1} \hat{R}(z) S \qquad \text{for all } z \in \Sigma. \]
	\end{equiv_enum}
	Moreover, if $L_p(\Omega)$ is separable, then so is $L_p(\tilde{\Omega})$.
\end{corollary}
	
\begin{proof}
	Note that the boundedness of an $H^{\infty}(\Sigma_{\theta})$-calculus passes through invariant subspace-quotients and is preserved by similarity transforms. Hence, \ref{char_hinfty_ii}~implies~\ref{char_hinfty_i} as every bounded analytic semigroup on some $L_p$-space %
	which is contractive and positive on the real line possesses a bounded $H^{\infty}$-calculus of angle lesser than $\frac{\pi}{2}$ by a result of Weis~\cite[Remark~4.9c)]{Wei01}.
	
	Conversely, \ref{char_hinfty_i} implies \ref{char_hinfty_ii} by Corollary~\ref{cor:generic_lp}.
\end{proof}

\begin{remark}
	It seems to be unknown whether the result of Weis holds on general UMD-Banach lattices. If this question has a positive answer, then the above characterization would also hold for general UMD-Banach lattices.
\end{remark}

\bibliographystyle{amsalpha}
\bibliography{Literatursammlung}{}

\end{document}